\documentclass[12pt]{amsart}
\usepackage{amssymb}
\usepackage{float}
\usepackage{todonotes}
\restylefloat{table}
\newcommand{\QQ}{\mathbb{Q}}
\newcommand{\ZZ}{\mathbb{Z}}
\newcommand{\fp}{\mathfrak{p}}
\newcommand{\fq}{\mathfrak{q}}
\newcommand{\fI}{\mathfrak{I}}

\newcommand{\fn}{\mathfrak{n}}
\newcommand{\OK}{\mathcal{O}_K}
\newcommand{\rmt}{\sqrt[3]{-14}}

\newcommand{\rmpic}{\sqrt{-1500}}
\newcommand{\Ek}{E_{\kappa}}
\newcommand{\brhoEk}{\bar{\rho}_{\Ek,n}}

\DeclareMathOperator{\rank}{rank}
\DeclareMathOperator{\tor}{tor}
\DeclareMathOperator{\Sel}{Sel}
\DeclareMathOperator{\pr}{pr}

\def\mod#1{{\ifmmode\text{\rm\ (mod~$#1$)}
\else\discretionary{}{}{\hbox{ }}\rm(mod~$#1$)\fi}}

\newtheorem{theorem}{Theorem}[section]

\newtheorem{proposition}[theorem]{Proposition}

\theoremstyle{remark}

\numberwithin{equation}{section}

\def\mod#1{{\ifmmode\text{\rm\ (mod~$#1$)}
\else\discretionary{}{}{\hbox{ }}\rm(mod~$#1$)\fi}}

\usepackage{hyperref}

\begin{document}
\title{The equation $(x-d)^5+x^5+(x+d)^5=y^n$}

\author{Michael A. Bennett}
\address{Department of Mathematics, University of British Columbia, Vancouver, BC Canada V6T 1Z2}
\email{bennett@math.ubc.ca}
\thanks{Supported in part by a grant from NSERC}

\author{Angelos Koutsianas}
\address{Department of Mathematics, University of British Columbia, Vancouver, BC Canada V6T 1Z2}
\email{akoutsianas@math.ubc.ca}

%\thanks{Supported in part by a grant from NSERC}
%\subjclass{Primary 11D25, 11J86}

\date{\today}

%-----------------------------------------------------------------------------
\begin{abstract}
In this paper, we solve the equation of the title under the assumption that $\gcd(x,d)=1$ and $n\geq 2$. This generalizes  earlier work of the first author, Patel and Siksek \cite{BennettPatelSiksek16}. Our main tools include Frey-Hellegouarch curves and associated modular forms, and an assortment of Chabauty-type techniques for determining rational  points on curves of small positive genus.
\end{abstract}

\maketitle

%-------------------------------------------
\section{Introduction}
%-------------------------------------------

%Let us denote by 
%$$
%S_j(z,d,k)= \sum_{i=0}^{j-1} (z+id)^k
%$$
%the sum of the $k$th powers of $j$ consecutive terms in an arithmetic progression. Properties of these objects have been long-studied.

 If $f(x)$ is a polynomial in $\mathbb{Z}[x]$ of degree at least three with, say, distinct complex roots, then the {\it superelliptic equation}
 $$
 f(x)=y^n
 $$
 has at most finitely many solutions in integers $x, y$ and $n$ with $|y|, n \geq 2$. This result, due to Schinzel and Tijdeman \cite{SchinzelTijdeman76},  is a consequence of bounds for linear form in logarithms and, as such, is effective -- one can readily deduce explicit bounds upon $y$ and $n$, depending only upon $f$. Actually classifying these solutions for a given polynomial $f(x)$, however, is, in most cases, still beyond current technology. By way of example, while we strongly suspect  that the equation
 $$
 x^2-2=y^n
 $$
 has only the obvious solutions with $y=-1$, this has not yet been proven.
 
 If we replace the polynomial $f(x)$ with a binary form $f(x,d)$ of degree $k \geq 3$, we would expect that analogous finiteness statements hold for the {\it generalized superelliptic equation}
\begin{equation} \label{GSE}
 f(x,d)=y^n, \; \; \gcd(x,d)=1,
\end{equation}
at least provided $n$ is suitably large relative to $k$. That such a result might be inherently difficult to prove  in any generality is suggested by the fact that the case $f(x,d)=xd(x+d)$ is essentially equivalent to Fermat's Last Theorem \cite{Wiles95}. Some statements for forms of degree $k \in \{ 3, 4, 6, 12 \}$ may be found in a paper of the first author and Dahmen \cite{BennettDahmen13}, though they are applicable to a vanishingly small proportion of such forms unless $k=3$.

One accessible source of  $f(x,d)$ for which equation (\ref{GSE}) has proven occasionally tractable is provided by forms derived from sums of powers of consecutive terms in arithmetic progression.  Defining
$$
S_j(x,d,k)= \sum_{i=0}^{j-1} (x+id)^k,
$$
equation (\ref{GSE}) with $f(x,d)=S_j(x,d,k)$ is a long-studied problem
(see e.g. \cite{Schaffer56,VoorhoeveGyoryTijdeman79,GyoryTijdemanVoorhoeve80,Brinza84,Urbanowicz88,Pinter97,	BennettGyoryKalmanPinter04,Pinter07,Hajdu15,PatelSiksek17,Argaez19}). The case of three terms, i.e. $j=3$, 
has attracted particular attention over the years, dating back to work of Cassels and Uchiyama independently \cite{Cassels85,Uchiyama79}, who studied the case $d=1$, $n=2$ and $k=3$. More recently,
all solutions of the equation
\begin{equation} \label{three}
S_3(x,d,k)=y^n
\end{equation}
have been determined for $2 \leq k \leq 7$, and various values of $d$ (where the prime divisors of $d$ are small and fixed). The interested reader may consult \cite{Zhang14,BennettPatelSiksek16,KoutsianasPatel18,Koutsianas19,Zhang17,Argaez19, ArgaezPatel19, ArgaezPatel20, DasKoutsianas20}.

Closely related to the focus of the paper at hand is recent work of van Langen \cite{Langen19}, where equation (\ref{GSE}) is completely resolved in the case $f(x,d)=S_3(x,d,4)$, through careful consideration of two Frey-Hellegouarch $\mathbb{Q}$-curves.
We will deduce an analogous result for 
$$
f(x,d)=S_3(x,d,5),
$$
though our argument is essentially rather more straightforward, utilizing elliptic curves defined over $\mathbb{Q}$.
We prove the following theorem.

\begin{theorem}\label{thm:main}
The only solutions to the equation
$$
(x-d)^5+x^5+(x+d)^5 = y^n,~x,y,d,n\in\ZZ,~n\geq 2,
$$
with $\gcd(x,d)=1$ satisfy $x=0, |d|=1$ or  $|x|=1, |d|=2$.
\end{theorem}

We note that this result generalizes an earlier theorem of the first author, Patel and Siksek \cite{BennettPatelSiksek16}, who treated the same equation under the restriction that $d=1$. The main new ingredients we employ, in comparison to \cite{BennettPatelSiksek16}, are a variety of Chabauty-type techniques to handle the small exponent cases $n \in \{ 2, 3, 5 \}$, which reduce to analyzing rational points on curves of positive genus, and an elementary observation about local behaviour at $p=3$.

For each of the four cases $n=2, 3, 5$ and $n \geq 7$, we will make use of the computational algebra package Magma \cite{Magma}. The relevant code for all the computations in this paper can be found at 
\begin{center}
\url{https://github.com/akoutsianas/ap_5th_powers}.
%\url{https://sites.google.com/site/angeloskoutsianas/ap_5th_powers.7z?attredirects=0&d=1}.
\end{center}

%-------------------------------------------
\section{Preliminaries}
%-------------------------------------------

Here and henceforth, let us suppose that we have
\begin{equation}\label{eq:main}
(x-d)^5+x^5+(x+d)^5 = y^n,
\end{equation}
with $x$ and $d$ coprime nonzero integers, and $n\geq 2$ an integer. Without loss of generality, we may suppose that $n$ is prime.
Then
$$
x ( 3x^4+20x^2d^2+10d^4)=y^n
$$
and so, writing $\kappa = \gcd (x,10)$, there necessarily exist integers $a$ and $b$ such that
\begin{equation}\label{eq:x_power}
x = \kappa^{n-1} a^n
\end{equation}
and
\begin{equation} \label{fish}
 3x^4+20x^2d^2+10d^4 = \kappa b^n.
\end{equation}
Here, $\kappa a$ and $b$ are coprime. If $xy=0$, then from (\ref{eq:main}), $x=y=0$. We will thus suppose henceforth that $a \neq 0$ and rewrite the last equation as
\begin{equation}\label{eq:pp2_a}
10(d^2 + x^2)^2 - 7x^4 = \kappa b^n.
\end{equation}
From  \eqref{eq:x_power} and \eqref{eq:pp2_a}, we are led to the Diophantine equation
\begin{equation} \label{fowl}
7 \kappa^{4n-5} a^{4n} +  b^n  = \left( \frac{10}{\kappa} \right) T^2,
\end{equation}
where $T = d^2 + x^2$. A simple but very helpful observation is as follows. Since $\gcd(x,d)=1$ and $n \geq 2$, equation (\ref{fish}) implies that $3 \nmid d$.
It follows from (\ref{fish}), then, that 
$$
\kappa b^n \equiv 1-x^2 \mod{3}.
$$
Thus either $3 \mid x$ (in which case, from (\ref{eq:x_power}), $3 \mid a$), or $3 \nmid x$, so that $x^2 \equiv 1 \mod{3}$, whence $3 \mid b$. 
We thus may conclude, in all cases, that $3 \mid ab$ (i.e. that $3 \mid y$ in equation (\ref{eq:main})).

%---------------------------------
\section{The case $n \geq 7$}
%--------------------------------

We begin by treating the case of larger prime exponents, i.e. those with $n \geq 7$.
Following \cite{BennettSkinner04}, we construct signature $(n,n,2)$ Frey curves $E_\kappa$ for each value of $\kappa$, namely 
\begin{equation} \label{F1}
E_1 \; \; : \; \; Y^2 = X^3 + 20TX^2 + 10b^nX, 
\end{equation}
\begin{equation} \label{F2}
E_2 \; \; : \; \; Y^2 + XY = X^3 + \frac{5T - 1}{4}X^2 + 35\cdot 2^{4n-11}a^{4n}X, 
\end{equation}
\begin{equation} \label{F3}
E_5 \; \; : \; \; Y^2 = X^3 + 4TX^2 + 2b^nX
\end{equation}
and
\begin{equation} \label{F4}
E_{10} \; \; : \; \; Y^2 + XY = X^3 + \frac{T-1}{4}X^2 +  7\cdot 10^{4n - 11}a^{4n}X.
\end{equation}

We denote by $\brhoEk$ the Galois representation acting on the $n$--torsion subgroup of $\Ek$. For prime $n\geq 7$, from modularity \cite{Wiles95,TaylorWiles95,BreuilConradDiamondTaylor01}, irreducibility of $\brhoEk$ \cite[Corollary 3.1]{BennettSkinner04}, the fact that $3 \mid ab$ (so that, in particular, $ab \neq \pm 1$), and Ribet's level lowering theorem \cite[Theorem 1.1]{Ribet90}, we may conclude that  there exists a  weight $2$ cuspidal newform $f$ of level $N_{E\kappa}$, 
with $q$-expansion
\begin{equation} \label{frog}
f = q + \sum_{i\geq 2}a_i(f)q^i
\end{equation}
and eigenvalue field $K_f$, such that  $\Ek \sim_n f$.
Here,
$$
N_{E_\kappa} = \left\{
\begin{array}{cc}
2^8 \cdot 5^2 \cdot 7 & \mbox{ if } \kappa = 1,\\
2\cdot 5^2 \cdot 7 & \mbox{ if } \kappa = 2,\\
2^8 \cdot 5 \cdot 7 & \mbox{ if } \kappa = 5,\\
2 \cdot 5 \cdot 7 & \mbox{ if } \kappa = 10,\\
\end{array}
\right.
$$
and, by $\Ek \sim_n f$, we mean that there exists a prime ideal $\fn \mid n$ of $K_f$ such that, for almost all primes $p$, 
$$
a_p(E_\kappa)\equiv a_p(f)\mod{\fn}.
$$

In order to bound $n$ we appeal to the following by-now  standard result.

\begin{proposition} \label{goblet}
Suppose $E/\QQ$ is an elliptic curve of conductor $N_E$ and $f$ is a newform of weight $2$, level $N_f\mid N_E$, and $q$-expansion as in (\ref{frog}).
Suppose $E\sim_n f$ for some prime $n$. Then there exists a prime ideal $\fn\mid n$ of $K_f$ such that, for all primes $p$:
\begin{itemize}
\item if $p\nmid nN_EN_f$ then $a_p(E)\equiv a_p(f)\mod{\fn}$,
\item if $p\nmid nN_f$ and $p \ \| \ N_E$ then $\pm(p+1)\equiv a_p(f)\mod{\fn}$.
\end{itemize}
\end{proposition}

It is important to note here that, {\it a priori}, this result does not automatically provide an upper bound upon $n$, as we might have that $a_p(E)=a_p(f)$ for all $p$. In our case, however, where we apply Proposition \ref{goblet} with $E=E_\kappa$,  the fact that $3 \mid ab$ implies that $3 \nmid nN_f$ and $3 \mid N_{E_\kappa}$, so that 
\begin{equation} \label{goody}
a_3(f) \equiv \pm 4 \mod{\mathfrak{n}}.
\end{equation}
This congruence alone eliminates many possible forms $f$ from consideration and, in all cases, provides an upper bound upon $n$. Explicitly, we find after computing representatives for each Galois conjugacy class of weight $2$ newforms of levels $70, 350, 8960$ and $44800$ (there are $1$, $8$, $64$ and $196$ such classes, respectively), that $n \leq 16547$ (with the value $n=16547$ corresponding to a pair of classes of forms of level $44800$).
For the remaining pairs $(f,n)$, we may then appeal to Proposition \ref{goblet} with $E=E_\kappa$ and use the fact that $E_\kappa$ has nontrivial rational $2$-torsion (so that $a_p(E_\kappa) \equiv 0 \mod{2}$, for each prime $p$ of good reduction). For each prime $p \geq 11$,  $p \neq n$, Proposition \ref{goblet} thus implies a congruence of the shape
\begin{equation} \label{conger}
a_p(f)\equiv 2 m \mod{\fn} \; \mbox{ or } \; a_p(f)\equiv \pm (p+1) \mod{\fn},
\end{equation}
where $\fn \mid n$ and, by the Weil bounds, $m$ is an integer with $|m| \leq \sqrt{p}$. Applying these congruences to each remaining pair $(f,n)$ for each prime $11 \leq p \leq 97$, $p \neq n$, we eliminate all but a number of pairs with $n=7$ ($2$ forms at level $350$, $4$ forms at level $8960$ and $30$ forms of level $44800$), $n=11$ ($12$ forms of level $44800$), and $n=13$ ($4$ forms at level $8960$ and $8$ forms of level $44800$).

To show that these remaining $60$ pairs $(f,n)$ fail to give rise to solutions to equation (\ref{fowl}), 
we will appeal to Proposition \ref{goblet} with more carefully chosen primes $p$. 
This argument has its genesis in work of  Kraus \cite{Kraus98}. The key observation is that if $p \equiv 1 \mod{n}$ is not too large, relative to $n$, then $n$-th powers take on relatively few values modulo $p$ and hence, from (\ref{fowl}) and the models for the $E_\kappa$, the choices for $m$ in (\ref{conger}) are greatly restricted. An extreme example for this is the case $n=11$ and $\kappa=1$, where, choosing $p=23$, we find that (\ref{fowl}) has no solutions modulo $23$, unless $23 \mid ab$.
There are thus no values of $m$ to consider in (\ref{conger}) and hence
\begin{equation} \label{bad}
a_{23}(f)\equiv \pm 24 \equiv \pm 2 \mod{\fn},
\end{equation}
for a prime $\fn$ in $K_f$ with $\fn \mid 11$. This eliminates $8$ of the remaining  $12$ pairs $(11,f)$.
For the other $4$ pairs, (\ref{bad}) cannot be ruled out and we must choose a different prime $p$. With, for example, $p=89$, we find that for a given integer $x$, either $89 \mid x$ or 
$$
x^{11} \equiv \pm 1, \pm 12, \pm 34, \pm 37 \mod{89}.
$$
From (\ref{fowl}), it follows that either $89 \mid ab$, or that the pair $(b^{11},T) $ is congruent to one of 
$$
(1, \pm 28), (-1, \pm 27), (12, \pm 32), (-12, \pm 20), (37, \pm 29), (-37, \pm 7),
$$
modulo $89$. A short calculation with the model (\ref{F1}) reveals that either $89 \mid ab$ or
$$
a_{89} (E_1) \in \{ -4, 12, 14 \},
$$
i.e. that $m$ in (\ref{conger}) is restricted to the set  $\{ -2, 6, 7 \}$.
We thus have that
$$
a_{89} (f) \equiv -24, -4, 12, 14, 24 \equiv 1, \pm 2, 3, -4   \mod{\fn},
$$
for a prime $\fn$ in $K_f$ with $\fn \mid 11$. We arrive at the desired contradiction upon calculating that, for the $4$ forms in question, $a_{89} (f) \equiv 0  \mod{\fn}$.

More formally, we can state this argument as follows. Let $n\geq 7$ and $p=nt+1$ be prime, and define
$$
\mu_n(\mathbb{F}_p) = \{x_0\in\mathbb{F}_p^*~:~x_0^t=1\}.
$$
For a pair $(a,b)\in\mu_n(\mathbb{F}_p)$, let $T_{a,b}\in\mathbb{F}_p$ be such that 
$$
\frac{10}{\kappa}T^2_{a,b} = 7\kappa^{4n-5}a^4 + b.
$$
For a triple $(a,b,T_{a,b})$ we denote by $E_{\kappa,a,b,T_{a,b}}$ the elliptic curve over $\mathbb{F}_p$ derived from $E_\kappa$ by replacing $(a^n,b^n,T)$ with $(a,b,T_{a,b})$.

\begin{proposition}\label{prop:Kraus}
With the above notation, suppose that $f$ is a newform with $n\nmid \mbox{Norm}_{K_f/\QQ}(4 - a_p^2(f))$. If, for all triples $(a,b,T_{a,b})$, we have that 
$$
n\nmid \mbox{Norm}_{K_f/\QQ}(a_p(f) - a_p(E_{\kappa,a,b,T_{a,b}})),
$$ 
then no solutions to equation \eqref{fowl} arise from $f$.
\end{proposition}

\begin{proof}
The proof is similar to \cite[Proposition 8.1]{BugeaudMignotteSiksek06}.
\end{proof}

Applying Proposition \ref{prop:Kraus} with $p \in \{ 29, 43 \}$ (if $n=7$), $p \in \{ 23, 89 \}$  (if $n=11$) and $p \in \{ 53, 79, 157 \}$ (if $n=13$), completes the proof of Theorem \ref{thm:main} in case $n \geq 7$ is prime. The above computations can be found at
\begin{center}
\url{https://github.com/akoutsianas/ap_5th_powers}
\end{center}
 in the file \textit{elimination\_step.m}. We note that there are two methods one can use to compute data about classes of weight $2$ modular forms in Magma, either the classical approach, or by viewing such forms as Hilbert modular forms over $\mathbb{Q}$. In our situation, they are both reasonable computations (taking less than $24$ hours on a 2019 Macbook Pro); the Hilbert modular form package is substantially faster in the current Magma implementation. It is perhaps worth observing that the labels of forms generated by these two methods do not match.

%a short Magma computation reveals that there are no solutions to equation \eqref{fowl} for $n\geq 7$ unless $\kappa=1$ and $n=11$ for $4$ newforms $f$ ($i=175,176,189,190$ in my notation!).

%---------------------------------
\section{The case $n=2$}
%--------------------------------

In case $n=2$, we find immediately that  equation (\ref{fish}) has no integer solutions with  $\kappa = 1$. From (\ref{eq:pp2_a}), working modulo $8$, it follows that necessarily
 $b,d$ are odd, $x$ is even and $\kappa \in \{ 2, 10 \}$. 
Since then
$$
3x^4+20x^2d^2+10d^4 \equiv 10 \mod{16},
$$
we thus have $\kappa=10$, i.e.
\begin{equation} \label{fix}
b^2 = d^4 + 200 d^2 a^4 + 3000a^8.
\end{equation}
Writing
$$
y= \frac{4d(b+d^2+100a^4)}{a^6} \; \mbox{ and } \; 
x=\frac{2(b+d^2+100a^4)}{a^4},
$$
we find that
$$
y^2 = x^3 -400x^2+28000x.
$$
The latter equation defines a model for an elliptic curve, given as 
134400ed1 in Cremona's  database, of rank $0$ over $\mathbb{Q}$ (with torsion subgroup of order $2$). It follows that necessarily either $ad=0$, or that $b+d^2+100a^4=0$. The first cases are excluded by assumption. If $b+d^2+100a^4=0$, then, from (\ref{fix}), again we have that $a=0$, a contradiction.

%---------------------------------
\section{The case $n=3$}
%--------------------------------

If we next assume that $n=3$, equations (\ref{eq:x_power}) and (\ref{fish}) imply that
\begin{equation}\label{eq:Picard_curve_n_3}
b^3=\left( \frac{10}{\kappa} \right) d^4 + 20 \kappa^3 d^2 a^6 + 3 \kappa^{7} a^{12}
\end{equation}
and so
\begin{equation} \label{mule}
b^3 = \left( \frac{10}{\kappa} \right) \left( d^2 + \kappa^4 a^6 \right)^2 - 7 \kappa^{7} a^{12},
\end{equation}
i.e. 
\begin{equation} \label{Ellie}
Y^2= X^3 + 7 \kappa  \left( \frac{10}{\kappa} \right)^3,
\end{equation}
where
$$
Y= \frac{100 (d^2+\kappa^4a^6)}{\kappa^5 a^6} \; \mbox{ and } \; X= \frac{10b}{\kappa^3a^4}.
$$
In case $\kappa \in \{ 1, 10 \}$, the elliptic curve corresponding to (\ref{Ellie}) has Mordell-Weil rank $0$ over $\mathbb{Q}$ and trivial torsion. Since we have assumed that $a \neq 0$, we may thus suppose that $\kappa \in \{ 2, 5 \}$. In these cases, the curves corresponding to (\ref{Ellie}) have Mordell-Weil rank $1$ over $\mathbb{Q}$; we will tackle these values of $\kappa$ with two different approaches. 

\vspace{0.5cm}
\paragraph{\textbf{The case $\kappa=2$.}} For this choice of $\kappa$,  from equation \eqref{mule} we have
\begin{equation}\label{eq:k_2_param}
5Z^2 = X^3 + 14Y^3,
\end{equation}
where $Z = d^2 + 2^4a^6$, $X = b$ and $Y = 2^2a^4$. It follows that $\gcd (X,Y)=1$. Suppose $K = \QQ(\rmt)$; the class number of $K$ is $3$, $\OK=\ZZ[\rmt]$ and a generator of the free part of the unit group of $\OK$ is given by
$$
r = -1 + 2\rmt + (\rmt)^2. 
$$
We factorize \eqref{eq:k_2_param} over $K$ to obtain
\begin{equation}\label{eq:k_2_factorization}
5Z^2 = (X - \rmt Y) (X^2 + \rmt XY + (\rmt)^2 Y^2).
\end{equation}
There are unique prime ideals of $K$ above $2$, $3$ and $7$ which we denote by $\fp_2$, $\fp_3$ and $\fp_7$, respectively. If $\fp$ is an prime ideal of $K$ that divides $\gcd(X - \rmt Y,X^2 + \rmt XY + (\rmt)^2 Y^2)$, then we can readily conclude that $\fp\mid (3\rmt X,3\rmt Y)$. Because $\gcd (X,Y)=1$, it follows that $\fp\mid 3\rmt$, whence $\fp=\fp_2$, $\fp_3$ or $\fp_7$. There are two prime ideals of $K$ above $5$ with norm $5$ and $5^2$ which we denote by $\fp_5$ and $\fq_5$, respectively; we have  $\langle 5\rangle = \fp_5\fq_5$. From the above discussion, we may write
\begin{equation}\label{eq:k_2_fac_ideal}
\langle X - \rmt Y\rangle = \fp_2^{a_2}\fp_3^{a_3}\fp_7^{a_7}\fp_5^{a_5}\fq_5^{a^\prime_5}\fI^2,
\end{equation}
where $0\leq a_i < 2$ for $i=2,3,7$, $(a_5,a^\prime_5) = (1,0)$ or $(0,1)$ and $\fI$ is an integral ideal of $\OK$. Because $N(\fp_2)=2$, $N(\fp_3)=3$, $N(\fp_7)=7$ and $N(\fq_5)=5^2$, taking norms on both sides of equation \eqref{eq:k_2_fac_ideal} and appealing to \eqref{eq:k_2_param}, we necessarily have that $a_2=a_3=a_7=0$ and $(a_5,a^\prime_5)=(1,0)$. In conclusion, 
$$
\langle X - \rmt Y\rangle = \fp_5\fI^2.
$$
The ideal $\fp_5$ is not principal but $\fp_5^3 = \langle 5 - \rmt + (\rmt)^2\rangle$. Therefore, we have
\begin{equation} \label{fox}
(X - \rmt Y)^3 = \pm r^i(5 - \rmt + (\rmt)^2)\gamma^2,
\end{equation}
where $\gamma\in\OK$ and $i=0,1$. Replacing $(X,Y)$ be $(-X,-Y)$ we can only consider the positive sign in the last equality. Suppose that
$$
\gamma = s + t\rmt + u(\rmt)^2
$$
 with $s,t,u\in\ZZ$. After expanding the left and right hand sides of (\ref{fox}), we find that either
$$
\begin{cases}
X^3 + 14Y^3 = 5s^2 - 28st + 14t^2 + 28su - 140tu + 196u^2\\
-3X^2Y =-s^2 + 10st - 14t^2 - 28su + 28tu - 70u^2,\\
 3XY^2 =  s^2 - 2st + 5t^2 + 10su - 28tu + 14u^2,
\end{cases}
$$
or
$$
\begin{cases}
X^3 + 14Y^3 = -19s^2 - 56st + 42t^2 + 84su + 532tu + 392u^2,\\
-3X^2Y = -3s^2 - 38st - 28t^2 - 56su + 84tu + 266u^2,\\
 3XY^2 = 2s^2 - 6st - 19t^2 - 38su - 56tu + 42u^2,
\end{cases}
$$
if $i=0$ or $i=1$, respectively. Since $Y = 2^2a^4$ is even, in either case we have that $s$ is necessarily even and hence so is $X^3+14Y^3$, contradicting the fact that $\gcd (X,Y)=1$.

\vspace*{0.5cm}
\paragraph{\textbf{The case $\kappa=5$.}} Finally, let us suppose that $\kappa = 5$. A solution $(a,b)$ of \eqref{eq:Picard_curve_n_3}  corresponds to a rational point on the Picard curve
$$
C:~X_1^3 = Y_1^4 + 200\cdot 5^2 Y_1^2 + 3000 \cdot 5^4,
$$
where $X_1 = \frac{2b}{a^4}$ and $Y_1 = \frac{2d}{a^3}$. Let $J = Jac(C)$ be the Jacobian of $C$. Using Magma's  \textit{PhiSelmerGroup} routine, we can prove that $\rank(J)\leq 1$, while \textit{ZetaFunction} shows that $\# J(\QQ)_{\tor}=1$. On the other hand, we have the non-trivial point 
$$
P:=[(-150,\rmpic) + (-150,-\rmpic) - 2\infty]\in J(\QQ).
$$
Therefore, it follows that $J(\QQ)\cong \ZZ$ and that $P$ generates an infinite subgroup of $J(\QQ)$. Applying the code\footnote{The code can be found at \url{https://github.com/travismo/Coleman}.} from \cite{HashimotoMorrison20} with $p=19$ leads to the conclusion that $C(\QQ) = \{\infty\}$ which completes the proof for $n=3$.

%---------------------------------
\section{The case $n=5$}
%--------------------------------

From (\ref{fowl}), we are concerned, in all cases, with the genus $2$ hyperelliptic curve
$$
C:~ Y^2 = X^5 + 7 \cdot 10^5,
$$
where $Y=\frac{10^3(d^2 + x^2)}{\kappa^8a^{10}}$ and $X = \frac{10b}{\kappa^3a^4}$. Unfortunately, the rank of the Jacobian of $C$ is $3$ and Chabauty methods are not immediately applicable. Instead, we note that  the determination of $C(\QQ)$ can be reduced to the problem of computing all the points on a family of genus one curves over $\QQ(\sqrt[5]{7})$ with rational $X$-coordinates.

Let $K = \QQ(\theta)$ where $\theta^5 + 7=0$ and $\alpha = 10 \theta$. Then we can write
$$
X^5 + 7 \cdot 10^5 = (X - \alpha) \left( X^4 +\alpha X^3+ \alpha^2 X^2 + \alpha^3 X + \alpha^4 \right),
$$
from which it follows that 
$$
u^2 = \delta_j (X - \alpha) \; \; \mbox{ and } \; \;  v^2 = \delta_j \left( X^4 +\alpha X^3+ \alpha^2 X^2 + \alpha^3 X + \alpha^4 \right),
$$
where $u,v\in K^*$ and $\delta_j\in K^*/K^{*2}$. Since the polynomial $x^5 + 7\cdot 10^5$ is irreducible over $\QQ$, from \cite[Section 2]{BruinStoll09} we have that $\delta_j$ is an element of $\Sel^{(2)}(C/\QQ)\subset K^*/K^{*2}$. Using Magma's function \textit{TwoSelmerGroup} we can prove that $\Sel^{(2)}(C/\QQ)=\langle a_1,a_2,a_3\rangle$, where
\begin{align*}
a_1 & = -\theta^4 - \theta^3 + \theta^2 + \theta + 1,\\
a_2 & = \theta^4 - \theta^2 - 4\theta - 5,\\
a_3 & = -5\theta^4 - 7\theta^3 - 7\theta^2 - \theta + 15.
\end{align*}
We can therefore choose $\delta_j$ to be an $\mathbb{F}_2$-linear combination of $a_1,a_2$ and $a_3$, so that in particular, we may suppose that
\begin{align*}
\delta_1 & = 1,\\
\delta_2 & = -5\theta^4 - 7\theta^3 - 7\theta^2 - \theta + 15,\\
\delta_3 & = \theta^4 - \theta^2 - 4\theta - 5,\\
\delta_4 & = 75\theta^4 + 99\theta^3 + 73\theta^2 - 41\theta - 257,\\
\delta_5 & = -\theta^4 - \theta^3 + \theta^2 + \theta + 1,\\
\delta_6 & = -33\theta^4 - 65\theta^3 - 77\theta^2 - 49\theta + 43,\\
\delta_7 & = 9\theta^4 + 7\theta^3 - 3\theta^2 - 23\theta - 47,\\
\delta_8 & = 545\theta^4 + 913\theta^3 + 993\theta^2 + 381\theta - 1251.
\end{align*}
For each $\delta_j$, we have a corresponding genus one curve
\begin{equation}
C_j:~v^2 = \delta_j\left( X^4 +\alpha X^3+ \alpha^2 X^2 + \alpha^3 X + \alpha^4 \right).
\end{equation}
It is enough to determine the points $(X,v)$ of $C_j$ such that $X\in\QQ$. The Jacobian $J_j$ of $C_j$ is an elliptic curve over $K$ and using Magma we are able to compute a subgroup of $J_j(K)$ of full rank. The rank of $J_j(K)$ for each $j$ is less than $5$ (see Table \ref{table:rank_Jj}), hence $C_j$ satisfies the Elliptic Curve Chabauty condition over $K$ for every $j$, $1 \leq j \leq 8$ \cite{Bruin03}. If $j\neq 2$,  we are able to determine a $K$-point on $C_j$; in case  $j=2$, we are unable to find an element of  $C_2(K)$ and must argue somewhat differently.

\begin{center}
\begin{table}
\begin{tabular}{| c | c || c | c |} 
\hline
$j$ & $\rank(J_j(K))$ & $j$ & $\rank(J_j(K))$ \\
\hline
$1$ & $2$ & $5$ & $3$ \\
\hline
$2$ & $1$ & $6$ & $2$ \\
\hline
$3$ & $2$ & $7$ & $1$ \\
\hline
$4$ & $4$ & $8$ & $3$ \\
\hline
\end{tabular}
\caption{The rank of $J_j(K)$.}\label{table:rank_Jj}
\end{table}
\end{center}

\subsection{$C_j(K)\neq\emptyset$:} For the cases $j\neq 2$ we have computed a point $P_j\in C_j(K)$. Then it is known that $C_j$ is birationally equivalent to an elliptic curve $E_j$ over $K$ and the map $\phi_j: E_j\mapsto C_j$ may be found  explicitly in \cite[Theorem 1]{Connell92}. We also know that $E_j\simeq_K J_j$. Let $\pr:C_j\mapsto \bar{K}$ with $\pr(X,v) = X$ be the projection map and $\psi_j := \pr\circ\phi_j$. We observe that $\psi_j$ is defined over $K$. 

Let $(X,Y)$ be a rational point of $C$ that corresponds to a point $(X,v)$ on $C_j$. From the above discussion,  it is enough to determine $H_j = E_j(K)\cap\psi_j^{-1}(\QQ)$ because then $X\in\psi(H_j)$. As we have already mentioned, the Elliptic Curve Chabauty condition is satisfied for all cases and hence we can apply the method in \cite{Bruin03} together with Mordell-Weil sieve \cite{BruinStoll10}, as implemented in Magma, to determine $H_j$ (we also recommend \cite[Section 4]{BremnerTzanakis04} for a nice exposition of the method). For $j \in \{ 1, 3, 5, 6, 7, 8 \}$ it is enough to apply a Chabauty argument using the primes $\fp\mid p$ of good reduction of $E_j$ with $p\leq 47$, but for $j=4$ we are forced to consider $p\leq 199$. After carrying this out, we conclude, for $j\neq 2$, that  $X=-6$ or $30$.

\subsection{$C_j(K) = \emptyset$:} For the case $j=2$, we are not able to find a point $P\in C_2(K)$. However, we know that the curve $C_2$ is a $2$-covering of the elliptic curve $J_2$ which means that if $C_2(K)\neq \emptyset$ then $C_2$ corresponds to a non-trivial point of the $2$-Selmer group of $J_2$ over $K$ that lies on the image of $J_2(K)$.

Since $J_2(K)=\ZZ/2\ZZ\times\ZZ$, the image of $J_2(K)/2J_2(K)$ corresponds to three non-trivial elements of $\Sel^{(2)}(J_2/K)$ which in turn correspond to three $2$-covering curves of $J_2$ that have a $K$-point. Using Magma's function \textit{TwoDescent} we can compute all $2$-coverings of $J_2$ that are elements of $\Sel^{(2)}(J_2/K)$. We identify the three of them that have a $K$-point (using the function \textit{Points}) and note that none of them is isomorphic to $C_2$; we may thus  conclude that $C_2(K)=\emptyset$ and so no rational points of $C$ arise from $C_2$.

The data for the various $C_j$, $P_j$ and $E_j$, together with the corresponding Magma code can be found at
\begin{center}
\url{https://github.com/akoutsianas/ap_5th_powers}
\end{center}
 in the files \textit{ap\_5th\_powers\_data.m} and \textit{ap\_5th\_powers.m}.
 
To sum up, we have proved that 
$$
C(\QQ) = \{(30,\pm 5000),(-6,\pm 832),\infty\}.
$$ 
From the fact that $X = \frac{10b}{\kappa^3a^4}$ and $a \neq 0$, we thus have $b=3$ and $a=\kappa=1$, corresponding to the solutions to equation (\ref{eq:main}) with
$$
(|x|,|d|,|y|,n)=  (1, 2, 3, 5).
$$
This completes the proof of Theorem \ref{thm:main}.

%---------------------------------
\section{Concluding remarks}
%%--------------------------------

In this paper, we have studied  equation (\ref{GSE}) in the case $f(x,d)=S_3(x,d,5)$.
Important to our argument is that we are able to assert that $3 \mid y$, while, at the same time, the associated ternary equation (\ref{fowl}) has the property that its coefficients are coprime to $3$. A nearly identical approach enables one to treat (\ref{GSE}) if $f(x,d)=S_j(x,d,5)$ for every fixed $j \equiv \pm 3 \mod{18}$. Indeed, writing 
$$
x=z-\frac{(j-1)}{2} d, 
$$
we have that
$$
S_j(x,d,5) = \frac{jz}{3}  \left( 3z^4 +5 \left( \frac{j^2-1}{2} \right) z^2d^2 + \left( \frac{(j^2-1)(3j^2-7)}{16} \right)  d^4 \right).
$$
Once again equation (\ref{GSE}) and the fact that $n \geq 2$ implies that $3 \nmid d$ and hence
$$
y^n \equiv \pm z \left( 1 - z^2 \right)  \equiv 0 \mod{3}.
$$

%---------------------------------
\section{Acknowledgment}
%--------------------------------

The second author is grateful to Sachi Hashimoto, Travis Morrison and Prof. Michael Stoll for useful discussions and to Prof. John Cremona for providing access to the servers of the Number Theory group of Warwick Mathematics Institute where all the computations took place.

%%%%%%%%%%%%%%%%%%%%%%%%%%%%%%%%%%%%%%
\bibliographystyle{alpha}
\bibliography{my_bibliography}

\end{document}